\documentclass[leqno,11pt,a4paper]{amsart}

\usepackage{amssymb,latexsym}
\usepackage{graphicx}
\usepackage{hyperref}
\usepackage{amsmath}
\usepackage{amsfonts}
\usepackage{amssymb}
\usepackage{float}
\usepackage[shortlabels]{enumitem}
\theoremstyle{plain} 
\newtheorem{theorem}{\bf Theorem}[section]

\newtheorem{lemma}[theorem]{\bf Lemma}

\newtheorem{proposition}[theorem]{\bf Proposition}

\theoremstyle{definition} 
\newtheorem{definition}[theorem]{\bf Definition}
\newtheorem{remark}[theorem]{\bf Remark}

\title[The 4-intersection unprojection format]{The 4-intersection unprojection format}

\author[Vasiliki~Petrotou]{Vasiliki~Petrotou}
\address{VASILIKI~PETROTOU\\EINSTEIN INSTITUTE OF
MATHEMATICS, HEBREW UNIVERSITY OF JERUSALEM, 91904 JERUSALEM, ISRAEL}
\email{vasiliki.petrotou@mail.huji.ac.il}

\begin{document}

\subjclass[2010]{
Primary 14M05, 14J45; Secondary 13H10, 14E99.
}
\keywords{ Gorenstein rings, Fano 3-folds, Birational Geometry, Unprojection.
}

\begin{abstract}
Unprojection theory is a philosophy due to Miles Reid,  which becomes a useful tool in algebraic geometry 
for the construction and the study of new interesting geometric objects such as algebraic surfaces 
and 3-folds. In this present work we introduce a new format of unprojection, which we
call the 4-intersection format.  It is specified by a codimension $2$ complete intersection ideal 
$I$  which is contained in four codimension $3$ complete intersection ideals $J_1, J_2, J_3, J_4$
and leads to the construction of codimension $6$ Gorenstein rings. As an application, 
we construct three families of codimension $6$  Fano $3$-folds embedded in weighted projective 
space which correspond to the entries with identifier numbers $29376$, $9176$ and $24198$ respectively in 
the Graded Ring Database. 
\end{abstract}

\maketitle

\section{Introduction} 

\label{sec!introduction}

The theory of unprojection focuses on constructing and analyzing commutative rings in terms of simpler ones. 
It also provides, in an intrinsic way, the relation between the commutative rings associated to certain constructions which appear often 
in algebraic geometry such as for example the Castelnuovo blow-down of a $-1$-curve on a surface \cite{ P1, P2, PR, R1}
and also in algebraic combinatorics \cite{BP2}.

Kustin and Miller, motivated by the question of the structure
of  codimension $4$ Gorenstein rings \cite{KM1,KM2,KM3,KM4,KM5},  introduced \cite{KM}  a method that constructs Gorenstein 
rings with more complicated structure than the initial data, which is now known as 
Kustin-Miller  (or Type I) unprojection. In this unprojection type the codimension increases by $1$.
 
Some years later, around $1995$,  Reid reinterpreted and generalised 
the construction of Kustin and Miller and formulated the main principles of unprojection theory \cite{R1}. 
His motivation was to provide an algebraic language useful for the study of birational geometry.

Since then,  many papers on foundational questions of unprojection theory ~\cite{AL, CL1, CL2, LP, NP1, NP2, P2, P3, P4, PR, PV, TR} have appeared.

We now summarise Reid's formulation of unprojection.
Assume that $J\subset R$ is a codimension $1$ ideal with $R$, $R/J$ being Gorenstein. 
Denote by $i\colon J\rightarrow R$ the inclusion map. Then there exists 
an $R$-module homomorphism  $\phi \colon J\rightarrow R$ such that the 
$R$-module  $\operatorname{Hom}_R(J, R)$ is generated by the set $\{i, \phi\}$.
Using $\phi$, Reid defined the new unprojection ring, see \cite[ Definition~ 3.1]{PV}. 
Papadakis and Reid  proved that the ring of unprojection is Gorenstein (\cite[Theorem~ 1.5]{PR}). 
We refer the reader to \cite[ Example~ 2.3]{PR} for the simplest example of Kustin-Miller unprojection. 

Unprojection theory has found many applications. In birational geometry,  in the study of Fano 3-folds, 
algebraic surfaces of general type and Mori-flips ~\cite{BGR,BKR,CM,CPR,NP1}. 
Moreover, in the construction of some interesting geometric objects such as K3 surfaces, Fano $3$-folds 
and Calabi-Yau $3$-folds of high codimension ~\cite{BG, BKR, NP2, PV}. There are also interesting applications 
in algebraic combinatorics  \cite{BP1, BP2, BP3, BP4}.

Sometimes, especially for the construction of geometric objects of high codimension,  it is necessary to 
perfom not only one but a series of unprojections. 
Neves and Papadakis \cite{NP2} developed a theory which is based on the repeated use of Kustin-Miller 
unprojection and as a result produces Gorenstein rings of high codimension as through the process 
a new unprojection variable is added. This theory is called parallel Kustin-Miller unprojection. A brief summary 
of the main aspects of the theory is given in \cite[ Subsection~ 3.2]{PV}.

In this paper we develop a new format of unprojection, which we call the $4$-intersection format. 
It is specified by a codimension $2$  complete intersection ideal $I$ with the property that it is contained 
in four codimension $3$ complete intersection ideals $J_1,\dots,J_4$. Using this data 
we construct, by parallel Kustin-Miller unprojection, a codimension~$6$ Gorenstein ring. As an application we construct 
three families of Fano $3$-folds of codimension~$6$ embedded in weighted projective space which 
correspond to the entries with ID: 29376, ID: 9176,  and ID: 24198 in the Graded
Ring Database ~\cite{ABR,BR,BS1,BS2,BS3}.

In Section $2$ we give some preliminary notions and results that we need in this paper. In 
Section $3$ we introduce the $4$-intersection unprojection format. 
In Subsection \ref{subs!mainres2to6}, we give a specific example of $4$-intersection format, 
and we construct, using parallel unprojection, a codimension $6$ Gorenstein ring.

In Section $4$ we give some applications.  Subsections \ref{constr!29376}, \ref{constr2!9176} and  \ref{constr3!24198}
contain  three specific 
$4$-dimensional quotients of the ring studied in Section $3$. 
 We check, using partially the computer algebra systems  Macaulay2 \cite{GS}  and Singular \cite{GPS01},
 that the geometric objects defined by them
correspond to the above mentioned families of the Fano 3-folds.

\section{Preliminaries}
We start by recalling some notions and results that are required for the rest of the paper.
Denote by $k=\mathbb{C}$ the field of complex numbers. 

\begin{definition}\label{def!codim} 
Let $R$ be a Noetherian ring and  $I\subset R$ an ideal. We define the \textit{codimension} 
of $I$ in $R$, denoted  by  $\operatorname{codim} I$, as follows:
\[
       \operatorname{codim}  I =  \dim R - \dim R/I.
\]
\end{definition}

\begin{theorem}\label{thm!krullprincipal} (Krull's Principal Ideal Theorem)
Assume that $R$ is a local Noetherian ring and $I$ is an ideal of $R$ which 
is generated by $n$ elements. Then, $\operatorname{codim} I \leq n$.
\end{theorem}

For a proof of Theorem \ref{thm!krullprincipal}, see for example \cite[p.~414]{BH}.

In the present work,  we will call an ideal $I$ of a polynomial 
ring $k[x_1,\dots,x_n]$ over a field $k$ a \textit{complete intersection ideal} 
if $I$ can be generated by  $\operatorname{codim} I$ elements.
We refer to  \cite [Section~2.3]{BH} for more details about this notion.

\begin{definition} 
A Noetherian local ring $R$ is called  \textit{Gorenstein} if it has finite injective 
dimension as an $R$-module. More generally, a Noetherian ring $R$ is called Gorenstein  
if for every maximal ideal  $\mathfrak m$ of $R$ the localization $R_{\mathfrak m}$ 
is Gorenstein.
\end{definition}

\begin{definition} 
An ideal $I$ of a Gorenstein ring $R$ is called  \textit{Gorenstein} if the quotient ring $R/I$ is Gorenstein.
\end{definition}

\begin{theorem} (Serre) \label{thm!serbusc}
Let $R=k[x_1,\dots, x_n] / I$  be the polynomial ring in $n$ variables of positive degree divided by a homogeneous ideal $I$.
 If  $\text{codim I}=  1$ or  $2$  then
\begin{center}
 $R$ is Gorenstein $\Leftrightarrow$  $I$ is a complete intersection.
\end{center}
\end{theorem}

For a proof of Theorem \ref{thm!serbusc}, see for example  \cite[ Corollary~ 21.20]{E}.

\begin{definition} 
Let $R$ be a Gorentein local ring and $J\subset R$ a codimension $1$ ideal such that the quotient ring $R/J$ be Gorenstein. Under these assumptions, the $R$-module  $\operatorname{Hom}_R(J,R)$ is generated  by the inclusion map $i\colon J\rightarrow  R$ and an extra homomorphism $\phi \colon J\rightarrow  R$,  (\cite[ Lemma ~ 1.1]{PR}). Denote by $T$ a new unprojection variable. We call  \textit{Kustin-Miller unprojection ring},  $\operatorname{ Unpr}(J,R)$, of the pair $J\subset R$ the quotient 
\[ \operatorname{Unpr}(J,R) =\frac{R[T]}{(Tr-\phi(r): r \in J)}.\]
\end{definition}

\begin{theorem}(\cite[ Theorem~ 1.5]{PR})
The Kustin-Miller unprojection  ring $\operatorname{Unpr}(J,R)$  is Gorenstein.
\end{theorem}

More discussion about the motivation of study and basic examples of  Kustin-Miller unprojection are contained in \cite{ PR, PV, R1}.

For the main principles of parallel Kustin-Miller unprojection we  refer the reader to \cite{ NP2, PV}.

\subsection{Unprojection of a codimension~2 complete intersection inside a codimension~3 complete intersection} \label{fund!calc11}
In this subsection we specify a codimension~$2$ complete intersection ideal $I$ and a codimension~$3$ complete intersection $J$ such that $I\subset J$. Following 
\cite[Section~4]{P2}, we give the explicit description of the unprojection ring 
$\operatorname{Unpr}(J/I,R/I)$ of the pair $J/I\subset R/I$. 

Let $R=k[a_i, b_i, x_j]$, where $1\leq i\leq 3$ and $j\in \{1, 3, 5\}$, be the standard 
graded polynomial ring in $9$ variables over a field $k$.  We set 
\[
    f_1 =  a_1x_1+a_2x_3+a_3x_5,  \quad  \quad   f_2 = b_1x_1+b_2x_3+b_3x_5,
\]
and consider the  ideals 
\[
         I=(f_1, f_2),  \quad \quad  J=(x_1, x_3, x_5)
\] 
of $R$. We denote by $A$ the  $2\times 3$ matrix
\[ 
A =  \begin{pmatrix} 
        a_1 &  a_2 &  a_3\\
         b_1 &  b_2  & b_3
   \end{pmatrix}
\]
and, for $1 \leq i \leq 3$, by $A_i$ the $2\times 2$ submatrix of $A$  obtained by removing the $i$-th column of $A$.

\begin{proposition}
The ideal $I$ is a homogeneous codimension $2$  Gorenstein ideal of $R$ and the ideal $J$ is
a homogeneous codimension $3$  Gorenstein ideal. Moreover, $I$ is a subset of $J$.
\end{proposition}

\begin{proof}
We first prove that $\operatorname{codim} I = 2$. The ideal $I$ is generated by two 
homogeneous polynomials of $R$ of degree $2$. Hence, by Theorem \ref{thm!krullprincipal}
$\operatorname{codim} I\leq 2$. To prove the claim it is enough to show that $\operatorname{codim} I\geq 2$. 
We set  $f_3=-b_1f_1+a_1f_2$.  Let $>$ be the lexicographic order 
on $R$ with 
\[     
              a_1>\dots> a_3> b_1>\dots>b_3>x_1>\dots> x_3.
\]
We denote by $Q$ the initial ideal of $I$ with respect  $>$. It is well-known
that $\operatorname{codim} I = \operatorname{codim} Q$.

We set
\[
          L = ( a_1x_1,   b_1x_1,   a_1b_2x_3).
\]
Since the initial term of $f_1$ is $a_1x_1$,   the initial term of $f_2$ is $b_1x_1$ 
and the initial term of    $f_3$ is  $a_1b_2x_3$ we have $ L \subset Q$,
hence $\operatorname{codim}  L \leq \operatorname{codim} Q$.

We consider the affine variety $X=V(L) \subset \mathbb{A}^{9}$. It holds that 
\[ 
      X= V(x_1,x_3)\cup V(b_2, x_1)\cup V(a_1, b_1)\cup V(a_1, x_1),
\]
hence $\dim X = 9 - 2 = 7$.
Using that
\[
         \dim \ R/L=\dim \ X,
\]
it follows that $\operatorname{codim} L = 2$.   Hence
 $\operatorname{codim} I\geq 2$. Therefore, by Theorem  \ref{thm!serbusc} the ideal $I$ is Gorenstein.

We now prove that  $\operatorname{codim} J = 3$. According to the Third Isomorphism Theorem of rings 
\[
 R/J  \cong k[a_1 ,a_2, a_3, b_1, b_2, b_3]
\]
So, $\dim \ R/J= 6$. Hence,
\[
\operatorname{codim} \  J =  \dim \ R - \dim \ R/J = 3. 
\]
By the last isomorphism, the ideal $J$ is Gorenstein.
By the equality of matrices
\[
 \begin{pmatrix} 
f_{1}  \  f_{2} 
\end{pmatrix} =   A  \begin{pmatrix} 
x_{1} \\ x_{3} \\  x_{5} 
\end{pmatrix}
\]
it follows that $I\subset J$. 
\end{proof}

We set, for $1 \leq i \leq 3$, $h_i$ to be the determinant of the matrix $A_i$. Denote by 
\[\phi\colon J/I\rightarrow R/I\]
the map such that
\begin{center}
$\phi(x_1 + I)= h_1+ I, \, \, \, \,   \phi(x_3 + I)= -h_2+ I, \, \, \, \, \phi(x_5+ I)= h_3+ I$.
\end{center}
By \cite[Theorem~4.3]{P2}, $\operatorname{Hom}_{R/I}(J/I,R/I)$ is generated as $R/I$-module by the inclusion map $i$ and $\phi$. 
As a corollary,
\[ \operatorname{Unpr}(J/I,R/I) =\frac{R[T]}{I+(Tx_1-h_1, Tx_3-(-h_2), Tx_5-h_3)}.\]

\section{The $4$-intersection unprojection format} \label{sec!4intformat}

In this section we introduce the notion of  $4$-intersection unprojection format.

\begin{definition} Assume that $J_1,\dots,J_4$ are four codimension $3$ complete 
intersection ideals and $I$ is a codimension $2$ complete intersection ideal.  
We say that $I$ is a $4$-intersection ideal in $J_1, \dots , J_4$  if  
$I \subset J_t$ for all $1 \leq t \leq 4$.
\end{definition}

An important question is how to explicitly construct $I$ and $J_t$ such that 
$I$ is a $4$-intersection ideal in $J_1, \dots , J_4$.
In Subsection~\ref{subs!mainres2to6} we present such a construction.

\subsection{A specific $4$-intersection unprojection format} \label{subs!mainres2to6}

In the present subsection we specify the following: a codimension  $2$ complete intersection ideal $I$
and  four codimension~$3$ complete intersection ideals  $J_1,\dots ,J_4$ such that
$I$ is a $4$-intersection ideal in $J_1, \dots , J_4$.  Using this configuration 
as initial data, we construct,  by parallel Kustin-Miller unprojection~ \cite{NP2}, a 
codimension $6$ Gorenstein ring.

Assume that $k$ is a field. 
We consider the standard graded polynomial ring $R=~k[c_i, x_i]$, where $1\leq i \leq 6$. We set
\[
    f=c_1x_1x_2+c_2x_3x_4+c_3x_5x_6, \quad \quad  g=c_4x_1x_2+c_5x_3x_4+c_6x_5x_6, 
\]
$I=(f,g)$ and 
\[
      J_1= (x_1, x_3, x_5),  \; J_2= (x_1, x_4, x_6),   \;  J_3= (x_2, x_3, x_6),   \; J_4= (x_2, x_4, x_5).
\]
It is clear that  $f,g$ are homogeneous elements of degree $3$ and $I$ is a $4$-intersection ideal in the ideals $J_1, \dots , J_4$.

In the applications we need to specialize  the variables $c_i$  to 
elements of $k$.  We now give a precise way to do that. Consider the Zariski open subset
\[
      \mathcal{U}=\{(u_1,\dots, u_6)\in \mathbb{A}^{6}: u_i\neq 0 \, \, \, \text{for all} \, \,  \,1 \leq  i \leq 6 \}. 
\]
We assume   that  $(d_1, \dots , d_6)\in \mathcal{U}$. 
We denote by $\hat{R}=k[x_1,\dots, x_6]$ the polynomial ring in the variables  $x_i$. 
Let
     \[    \hat{\phi}\colon R\rightarrow \hat{R}\]
be the unique $k$-algebra homomorphism such that
\begin{center}
$\hat{\phi}(x_i)= x_i$, \, \,  $\hat{\phi}(c_i)= d_i$
\end{center}
for   all  $1 \leq  i \leq 6$. 
We denote by $\hat{I}$ the ideal of the ring $\hat{R}$ generated by the subset $\hat{\phi}(I)$.

\begin{proposition}  \label{prop!codim2}
The ideals $I$ and $\hat{I}$ are homogeneous codimension $2$ Gorenstein ideals.
\end{proposition}

\begin{proof}
Since $I$ is generated by two elements, we have,
by Theorem \ref{thm!krullprincipal},  that $\operatorname{codim} I\leq 2$. Now 
we show that $\operatorname{codim} I\geq 2$. We set
\[
        r_1= -c_4f+c_1g, \quad \quad   r_2= g,  \quad \quad  r_3=f . 
\]
 Let $>$ be the lexicographic order on $R$ with $c_1>\dots> c_6> x_1>\dots>x_6$. Consider the ideal 
\[ L =(\operatorname{in}_{>} (r_1), \operatorname{in}_{>} (r_2), \operatorname{in}_{>} (r_3)),\] where
$\operatorname{in}_{>}(r_1)= x_3x_4c_1c_5,\operatorname{in}_{>}(r_2)= x_1x_2c_4$  and  
 $\operatorname{in}_{>}(r_3)= x_1x_2c_1.$
We now prove that $\operatorname{codim} L = 2$. It is enough to show that $\dim \ R/L= 10$. 
Consider the affine variety $X=V(L)\subset \mathbb{A}^{12}$. It holds that 
\[ X= V(c_4,c_1)\cup V(c_5, x_1)\cup V(x_4, x_1)\cup V(x_3, x_1)\cup V(c_1, x_1)\cup V(c_5, x_2)\cup\]\[\cup V(x_4, x_2)\cup V(x_3, x_2)\cup V(c_1, x_2).\]
Using that,
\[\dim \ R/L=\dim \ X\]
the claim is proven. Hence, $\operatorname{codim} I\geq 2$. 

In what follows we show that the ideal $\hat{I}$ is also a codimension $2$ Gorenstein ideal. We set
\[\tilde{r_1}= \hat{\phi}(r_1), \quad \quad \tilde{r_2}= \hat{\phi}(r_2).  \]
 Let $>$ be the lexicographic order on $\hat{R}$ with $x_1>\dots> x_6$. Consider the ideal 
\[ 
          Q=(\operatorname{in}_{>} (\tilde{r_1}),  \operatorname{in}_{>} (\tilde{r_2})),
\] 
where
$\operatorname{in}_{>}(\tilde{r_1})= x_3x_4d_1d_5, \,  \operatorname{in}_{>}(\tilde{r_2})= x_1x_2d_4. $
It is immediate that  $Q=(x_3x_4, x_1x_2)$. It is enough to show that $\dim \ R/Q= 4$. 
Consider the affine variety $Y=V(Q)\subset \mathbb{A}^{6}$. It holds that 
\[
    Y= V(x_2,x_4)\cup V(x_2, x_3)\cup V(x_1, x_3)\cup V(x_1, x_4).
\]
Using that,
\[\dim \ R/Q=\dim \ Y\]
the claim is proven. Hence, $\operatorname{codim} \hat{I}\geq 2$. By Theorem \ref{thm!serbusc}, the ideals $I$ and $\hat{I}$ are Gorenstein. 
\end{proof}

\begin{proposition} \label{thm!assum1}
 
(i) For all $t$ with $1\leq t\leq 4$,  the ideal  $J_t / I$ is a codimension $1$ homogeneous ideal of the quotient ring $R/I$ such that the ring $R/ J_t$ is Gorenstein.

(ii) For all $t, s$ with $1\leq t < s\leq 4$,  it holds that  $\operatorname{codim}_{R/I}(J_t/I+J_s/I)= 3$.

\end{proposition}

\begin{proof}
We first prove $(i)$.
According to the Third Isomorphism Theorem of rings 
\begin{equation} \label{iso!pols11} 
 R/J_1  \cong k[c_1,\dots, c_6, x_2, x_4, x_6], \,  \, \,     R/J_2 \cong    k[c_1,\dots, c_6, x_2, x_3, x_5], 
\end{equation}
\begin{center}
$ R/J_3  \cong k[c_1,\dots, c_6, x_1, x_4, x_5], \,  \, \,     R/J_4 \cong    k[c_1,\dots, c_6, x_1, x_3, x_6].$
\end{center}
So, we conclude that for all t with $1\leq t\leq 4$,

\begin{center}
$ \dim \     R/J_t= 9.$
\end{center}
By Proposition \ref{prop!codim2}, it follows that
\begin{center}
$\dim \  R/I = \dim \  R- \operatorname{codim} \ I=10.$
\end{center}
Hence, using the last two equalities we have that for all t with $1\leq t\leq 4$
\[\operatorname{codim} J_{t}/I=1.\]
Due to the isomorphisms   (\ref{iso!pols11})  for all $t$ with $1\leq t\leq 4$,  the ring  $R/J_t$ is Gorenstein.

 Concerning the Claim $(ii)$, the Third Isomorphism Theorem of rings implies that

\[R/(J_1+J_2)  \cong k[c_1,\dots, c_{6}, x_2], \, \, \, \,  R/(J_1+J_3) \cong  k[c_1,\dots, c_{6}, x_4],  \] 
\[R/(J_1+J_4) \cong  k[c_1,\dots, c_{6}, x_6], \, \, \, \,  R/(J_2+J_3) \cong  k[c_1,\dots, c_{6}, x_5],\]
\[R/(J_2+J_4) \cong  k[c_1,\dots, c_{6}, x_3], \, \, \, \,  R/(J_3+J_4) \cong  k[c_1,\dots, c_{6}, x_1].\]
From the later isomorphisms it holds that for $t, s$ with $1\leq t < s\leq 4$,
\[\dim \  R/(J_t+J_s)= 7.\]

Recall that $\dim \  R/I= 10$. Taking into account the definition of codimension we conclude that for all $t, s$ with $1\leq t < s\leq 4$,
\[\operatorname{codim} \ (J_t/I+J_s/I)= 3.\] 
\end{proof}

For all $t$, with $1\leq t\leq 4$, we denote by $i_t\colon  J_t/I \rightarrow R/I$ the  inclusion map. In what follows, we  define  $  \phi_t\colon  J_t/ I \rightarrow R/I$ for all $t$, with $1\leq t\leq 4$, and  prove that these maps satisfy the assumptions of the  \cite[Theorem~2.3]{NP2}. 

Recall the polynomials $h_1, h_2, h_3$ which were defined in Section~\ref{fund!calc11}. We denote by $\widetilde{h_1},\widetilde{h_2},\widetilde{h_3}$  the polynomials which occur from $h_1, h_2, h_3$ if we substitute
\[a_1=c_1x_2, \,  a_2=c_2x_4, \,  a_3=c_3x_6, \,  b_1=c_4x_2, \,  b_2=c_5x_4, \,  b_3=c_6x_6. \]

\begin{proposition}\label{prop!defphi1}
There exists a unique graded homomorphism of $R/I$-modules \newline $\phi_1\colon J_1/ I \rightarrow R/I$ such that 
\begin{center}
$\phi_1(x_1 + I)= \widetilde{h_1}+ I, \, \, \, \,   \phi_1(x_3 + I)=\widetilde{h_2}+ I, \, \, \, \, \phi_1(x_5 + I)= \widetilde{h_3}+ I  $
\end{center}
\end{proposition}

\begin{proof}
It follows from \cite[Theorem~4.3]{NP2}.  
\end{proof}

For the definition of $\phi_2$ we replace $x_3$ by $x_4$ and $x_5$ by $x_6$. In this case,  $\widetilde{h_1},\widetilde{h_2},\widetilde{h_3}$ are the polynomials which occur from $h_1, h_2, h_3$ if we substitute
\[a_1=c_1x_2, \, a_2=c_2x_3, \, a_3=c_3x_5, \, b_1=c_4x_2, \, b_2=c_5x_3, \, b_3=c_6x_5. \] 
For the definitions of $\phi_3$ and  $\phi_4$ we work similarly. For all t, with $1\leq t\leq 4$, the degree of $\phi_t$ is equal to $3$.
By the discussion after \cite[Proposition~ 2.1]{NP2}  the new unprojection variable has degree equal to the degree of the corresponding $\phi_t$.

\begin{proposition}\label{prop!hom1}
For all $t$, with $1\leq t\leq 4$, the $R/I$-module $ \operatorname{Hom}_{R/I}(J_t/I,R/I)$ is generated by the two elements $i_t$ and $\phi_t$.
\end{proposition}

\begin{proof}
It follows from \cite[Theorem~4.3]{P2}.  
\end{proof}

For all $t, s$, with $1\leq t,s\leq 4$ and $t\neq s$, we define $r_{ts}=0$.

\begin{proposition} \label{prop!existrst1}
For all $t, s$, with $1\leq t,s\leq 4$ and $t\neq s$, it holds that
\[
 \phi_t(J_t/I) \subset  J_s/I.\]

\end{proposition}

\begin{proof}
It is a direct computation using the definition of the maps $\phi_t$. 
\end{proof}

\begin{proposition}
For all $t, s$, with $1\leq t,s\leq 4$ and $t\neq s$, there exists a homogeneous element $A_{st}$ such that
\[\phi_s (\phi_t (p)) = A_{st}p \]
for all  $p\in J_t/I$.

\end{proposition}

\begin{proof}
It follows from \cite[Proposition~2.1]{NP2}.  
\end{proof}

\begin{remark}
We note that the elements $A_{st}$ are polynomial expressions in the variables $c_i$ and $x_j$. We computed 
them using the computer algebra program Macaulay2~\cite{GS}. We now write down $A_{12}$:
\[A_{12}=(x_2^2)(c_3c_4-c_1c_6)(-c_2c_4+c_1c_5).\]
Applying symmetry, one can get formulas for all $A_{st}$.
\end{remark}

Following \cite[Section~2]{NP2}, we write down explicitly the final ring as a quotient of a polynomial ring by a codimension $6$ ideal.

\begin{definition} \label{def!mrin1}
Let  $T_1, T_2, T_3, T_4$ be four new variables of degree $3$. We define as  \textit{$I_{un}$} the  ideal 
\[  (I)+  (T_1x_1 \text{-} \phi_1(x_1),\,  T_1x_3 \text{-} \phi_1(x_3), \, T_1x_5 \text{-} \phi_1(x_5), \, T_2x_1 \text{-} \phi_2(x_1), \,  T_2x_4 \text{-} \phi_2(x_4), \, T_2x_6 \text{-} \phi_2(x_6),\] \[ T_3x_2 \text{-} \phi_3(x_2), \, T_3x_3 \text{-} \phi_3(x_3), \, T_3x_6 \text{-} \phi_3(x_6), \, T_4x_2 \text{-} \phi_4(x_2), T_4x_4 \text{-} \phi_4(x_4), \, T_4x_5 \text{-} \phi_4(x_5),\]\[T_2T_1 \text{-} A_{21}, \,  T_3T_1 \text{-} A_{31}, \, T_4T_1 \text{-} A_{41}, \, T_3T_2 \text{-} A_{32}, \,  T_4T_2 \text{-} A_{42}, \, T_4T_3 \text{-} A_{43} )\] 
of the polynomial ring  $R[T_1,T_2,T_3,T_4]$.
We set  $ \textit{$R_{un}$}= R[T_1,T_2,T_3,T_4]/ I_{un}$.
\end{definition}

\begin{remark}
The reason we put, for all $1 \leq i \leq 4$, $\operatorname{deg} T_i = 3$ is that each homomorphism $\phi_i$ 
is graded of degree $3$.  We also note that
according to \cite[Proposition~2.1]{NP2} the degree of each $A_{st}$ is equal to 6.
\end{remark}

\begin{theorem}\label{main!thm1}
The  ring  $R_{un}$ is Gorenstein.
\end{theorem}

\begin{proof}
 By Propositions  ~\ref{thm!assum1}, \ref{prop!defphi1} and  \ref{prop!existrst1},  the assumptions of  
\cite[Theorem~2.3]{NP2} are satisfied. Hence, the ring  $R_{un}$  is Gorenstein. 
\end{proof}

\begin{proposition}
The homogeneous ideal  $I_{un}$ is a codimension $6$ ideal with a minimal generating set of $20$ elements.
\end{proposition}

\begin{proof}
According to the grading of the variables and the discussion before Proposition~\ref{prop!hom1} it is not difficult to see that $I_{un}$ is a 
homogeneous ideal. Recall that in Kustin-Miller unprojection the codimension  increases by $1$. 
Hence, the homogeneous ideal $I_{un}$, as a result of a series of four unprojections of Kustin-Miller type starting by the
codimension $2$ ideal $I$, is a codimension $6$ ideal. In order to prove that $I_{un}$ is minimally
generated by $20$ elements we use the idea of specialization. More precisely we set
\begin{center}
$c_1=c_3=c_{5}=c_{6}=0$
\end{center}
and
\begin{center}
$c_2=c_4=1$
\end{center}
in the ideal  $I_{un}$. We call  $\widetilde{I_{un}}$  the ideal which occurs after these substitutions.
The ideal  $\widetilde{I_{un}}$  is a homogeneous 
ideal with $16$ monomials and $4$ binomials as generators. It is not 
difficult to see that $\widetilde{I_{un}}$ is minimally generated by these 
elements. Hence, we conclude that  $I_{un}$ is generated by at least $20$ elements. By 
Definition~\ref{def!mrin1}, $I_{un}$ is generated by $20$ homogeneous elements. The result follows. 
\end{proof}

\section{Applications}\label{Sec!app1}
In this section we prove,  using  Theorem \ref{main!thm1},  the existence of $3$ 
families of Fano $3$-~folds of codimension $6$ in weighted projective space. For some basic definitions and facts related to singularities and Fano $3$-folds which appear through this section we refer to  \cite[ Section~3]{PV}.
We note that in what follows we make essential use of the computer algebra systems Macaulay2 \cite{GS}  and Singular \cite{GPS01}.

 The first construction is summarised in the following theorem. It corresponds to the
entry $29376$ of Graded Ring Database~\cite{ABR,BR,BS1,BS2,BS3}. More details for the construction are given in Subsection~\ref{constr!29376}.

\begin{theorem}\label{constr!29376}
There exists a family of quasismooth, projectively normal and projectively 
Gorenstein Fano $3$-folds $X\subset \mathbb{P}(1^8, 2,3)$, nonsingular away from 
one quotient singularity $\frac{1}{3}(1,1,2)$, with Hilbert series

\begin{center}
$P_{X}(t)=\frac{1-6t^2+15t^4-20t^6+15t^{8}-6t^{10}+t^{12}}{(1-t)^8(1-t^2)(1-t^3)}.$
\end{center}
\end{theorem}

 \label{constr!11}

 The second construction is summarised in the following theorem. It corresponds to the
entry $9176$ of Graded Ring Database. More details for the construction are given in Subsection~\ref{constr2!9176}.

\begin{theorem}\label{constr!9176}
There exists a family of quasismooth, projectively normal and projectively 
Gorenstein Fano $3$-folds $X\subset \mathbb{P}(1^2, 2^5,3^3)$, nonsingular away from 
eight quotient singularities $\frac{1}{2}(1,1,1)$, with Hilbert series

\begin{center}
$P_{X}(t)=\frac{1-6t^4-8t^5+2t^6+24t^7+21t^8-16t^9-36t^{10}-16t^{11}+21t^{12}+24t^{13}+2t^{14}-8t^{15}-6t^{16}+t^{20}}{(1-t)^2(1-t^2)^5(1-t^3)^3}.$
\end{center}
\end{theorem}

 The third construction is summarised in the following theorem. It corresponds to the
entry $24198$ of Graded Ring Database. More details for the construction are given in Subsection~\ref{constr3!24198}.

\begin{theorem}\label{constr!24198}
There exists a family of quasismooth, projectively normal and projectively 
Gorenstein Fano $3$-folds $X\subset \mathbb{P}(1^6, 2^3,3)$, nonsingular away from 
two quotient singularities $\frac{1}{2}(1,1,1)$ and one quotient singularity $\frac{1}{3}(1,1,2)$, with Hilbert series

\begin{center}
$P_{X}(t)=\frac{1-t^2-10t^3+5t^4+24t^5-5t^6-28t^7-5t^8+24t^9+5t^{10}-10t^{11}-t^{12}+t^{14}}{(1-t)^6(1-t^2)^3(1-t^3)}.$
\end{center}
\end{theorem}

\subsection{ Construction of Graded Ring Database entry with ID: 29376} \label{constr!29376}
In this subsection, we give the details of the construction for the family described in Theorem \ref{constr!29376}.

Denote by $k=\mathbb{C}$ the field of complex numbers. Consider the 
polynomial ring $R= k[x_i, c_i]$, where $1\leq  i \leq 6$.  Let  $ R_{un}$ be the 
ring in Definition \ref{def!mrin1} and  $\hat{R}=k[x_1,\dots, x_6]$ be the polynomial ring in the variables $x_i$.
We substitute the variables $(c_1, \dots, c_{6})$ which appear in the definitions of the rings $R$ 
and $R_{un}$ with a general element of $k^{6}$ (in the sense of being outside a 
proper Zariski closed subset of $k^{6}$).   Let $\hat{I}$ be the ideal of $\hat{R}$ which is 
obtained by the ideal $I$ and $\hat{ I}_{un}$ the ideal of $\hat{R}[T_1,T_2,T_3,T_4]$ which is 
obtained by the ideal $I_{un}$  after this substitution. We set $\hat{R}_{un}= \hat{R}[T_1,T_2,T_3,T_4]/\hat{ I}_{un}$. In 
what follows $x_1, x_3, x_5$ are variables of degree $1$ and $x_2, x_4, x_6$ are variables of degree $2$. Hence, 
from the discussion before the Proposition \ref{prop!hom1} it follows that the 
degrees of $T_2, T_3, T_4$ are equal to $1$ and the degree of $T_1$ is equal to $3$. According to 
this grading the ideals $\hat{I}$ and $\hat{ I}_{un}$ are homogeneous. Due to 
Theorem~\ref{main!thm1},   $\text{Proj} \   \hat{R}_{un}\subset \mathbb{P} (1^{6}, 2^3, 3)$ is a projectively Gorenstein $3$-fold.

Let $A= k[w_{1}, w_{2}, T_2, T_3, T_4, x_1, x_3, x_5, x_6, T_1]$  be the polynomial ring 
over $k$ with $w_{1}, w_{2}$ variables of degree $1$ and the other variables of 
degree noted as above. Consider the unique $k$-algebra homomorphism
\begin{center}
$\psi\colon \hat{R}[T_1, T_2, T_3, T_4]\rightarrow A$
\end{center}
such that
\begin{center}
$\psi(x_1)= x_1$, \, $\psi(x_2)= f_1$, \,   $\psi(x_3)= x_3$, \, $\psi(x_4)= f_2$,
\end{center}
\begin{center}
$\psi(x_5)= x_5$, \, $\psi(x_6)= x_6$, \, $\psi(T_1)= T_1$, \,  $\psi(T_2)= T_2$,
\end{center}
\begin{center}
$\psi(T_3)= T_3$, \, $\psi(T_4)= T_4$
\end{center}
where,

$f_1= l_1x_1^2+ l_2x_1x_3+ l_3x_3^2+ l_4x_1x_5+ l_5x_3x_5+ l_6x_5^2+ l_7x_1T_2+ l_8x_{3}T_2+ l_9x_{5}T_{2}+  l_{10}T_{2}^2+ l_{11}x_{1}T_3+ l_{12}x_{3}T_3+ l_{13}x_{5}T_3+ l_{14}T_{2}T_3+ l_{15}T_3^2+l_{16}x_{1}T_4+
+ l_{17}T_{3}T_{4}+ l_{18}x_{5}T_4+l_{19}T_2T_4+l_{20}T_3T_4+l_{21}T_4^2+l_{22}x_1w_1+l_{23}x_3w_1+
l_{24}x_5w_1+l_{25}T_2w_1+l_{26}T_3w_1+l_{27}T_4w_1+ l_{28}w_1^2+l_{29}x_1w_2+l_{30}x_3w_2+l_{31}x_5w_2+
l_{32}T_2w_2+l_{33}T_3w_2+l_{34}T_4w_2+ l_{35}w_1w_2+l_{36}w_2^2+l_{37}x_6$,

$f_2= l_{38}x_1^2+ l_{39}x_1x_3+ l_{40}x_3^2+ l_{41}x_1x_5+ l_{42}x_3x_5+ l_{43}x_5^2+ l_{44}x_1T_2+ l_{45}x_{3}T_2+ l_{46}x_{5}T_{2}+  l_{47}T_{2}^2+ l_{48}x_{1}T_3+ l_{49}x_{3}T_3+ l_{50}x_{5}T_3+ l_{51}T_{2}T_3+ l_{52}T_3^2+l_{53}x_{1}T_4+l_{54}T_{3}T_{4}+l_{55}x_{5}T_4+l_{56}T_2T_4+l_{57}T_3T_4+l_{58}T_4^2+l_{59}x_1w_1+l_{60}x_3w_1+
l_{61}x_5w_1+l_{62}T_2w_1+l_{63}T_3w_1+l_{64}T_4w_1+ l_{65}w_1^2+l_{66}x_1w_2+l_{67}x_3w_2+l_{68}x_5w_2+
l_{69}T_2w_2+l_{70}T_3w_2+l_{71}T_4w_2+ l_{72}w_1w_2+l_{73}w_2^2+l_{74}x_6$,

\noindent and   $(l_1,\dots, l_{74})\in k^{74}$ are general. In other words, $f_1, f_2$ are two general degree $2$ homogeneous  elements of $A$.

Denote by $Q$ the ideal of the ring A generated by the subset   $\psi(\hat{I}_{un})$.

Let $X= V(Q)\subset \mathbb{P} (1^{8}, 2, 3)$. It is immediate that  $X\subset \mathbb{P}(1^8, 2, 3)$ is a codimension $6$ projectively Gorenstein $3$-fold.

\begin{proposition} \label{constr1!primeno2}
The ring  $A/Q$ is an integral domain.
\end{proposition}

\begin{proof}
It is enough to show that the ideal $Q$ is prime. For a specific choice of rational
values for the parameters $c_i,  l_j$, for $1\leq i\leq 6$ and $1\leq j\leq 74$ we checked, using the computer
algebra program Macaulay2, that the ideal which was obtained by specialization from $Q$  is a homogeneous, codimension~$6$, prime ideal with the right Betti table. 
\end{proof}

In what follows, we show that the only singularities of $X\subset \mathbb{P}(1^8, 2, 3)$  is a quotient singularity of type $\frac{1}{3}(1,1,2)$. According to the discussion after  
 \cite [Definition~2.7]{PV}, $X$ belongs to the Mori category. 

The proof of the following proposition is based on a computation with computer algebra system Singular \cite{GPS01} using the strategy described in \cite [Proposition~6.4]{PV} and is omitted.

\begin{proposition}\label{Prop!quasismooth1}
Consider   $X= V(Q)\subset \mathbb{P} (1^{8},2, 3)$. Denote by  $X_{cone}\subset \mathbb{A}^{10}$ the affine cone over $X$. 
The scheme  $X_{cone}$ is smooth outside the vertex of the cone.
\end{proposition}

\begin{remark}
For the computation of singular locus of weighted projective space in Proposition \ref{sing!specsing1},  we follow \cite [Section~5]{IF}. 
\end{remark}

\begin{proposition} \label{sing!specsing1}
Consider the singular locus 
\[
\text{Sing}( \mathbb{P}(1^{8}, 2, 3))= \{[0:0:0:0:0:0:0:0:1:0] \}\cup \{ [0:0:0:0:0:0:0:0:0:1]\} 
\]
 of the weighted projective space  $\mathbb{P}(1^8, 2,3)$. The intersection of $X$ with $\text{Sing}( \mathbb{P}(1^{8}, 2, 3))$ consists of a unique reduced point which is quotient singularity of type $\frac{1}{3}(1,1,2)$ for $X$.
\end{proposition}

\begin{proof}
We checked with the computer algebra program 
Macaulay2 that the intersection of $X$ with $Z$ consists of one reduced point. We
denote this point by $P$. The point $P$ corresponds to the ideal $(x_i, T_j , w_k)$ for $i\in \{1, 3, 5, 6\}$, $2\leq j\leq~4$, $1\leq k\leq 2$. By Proposition \ref{Prop!quasismooth1}, X is smooth outside $P$. Around $P$ we have that $T_1 = 1$.
Looking at the equations of $Q$ we can eliminate the variables $x_1, x_3, x_5, T_2, T_3, T_4$ since these
variables appear in the set of equations multiplied by $T_1$. This means that $P$ is a quotient
singularity of type $\frac{1}{3}(1,1,2)$. 
\end{proof}

\begin{lemma}\label{lem!canonmod1}
Let $\omega_{\hat{R}/\hat{I}}$ be the canonical module of $\hat{R}/\hat{I}$. It holds that  the canonical module $\omega_{\hat{R}/\hat{I}}$ is isomorphic to $\hat{R}/\hat{I}(-3)$.
\end{lemma}

\begin{proof}
From the minimal graded free resolution of $\hat{R}/\hat{I}$ as  $\hat{R}$-module
\[
0\rightarrow \hat{R}(-6)\rightarrow \hat{R}(-3)^{2}\rightarrow \hat{R}  
\]
and the fact that the sum of the degrees of the variables is equal to $9$ we conclude that 
\[\omega_{\hat{R}/\hat{I}}= \hat{R}/\hat{I}(6-9)= \hat{R}/\hat{I}(-3).\] 
\end{proof}

\begin{proposition} \label{prop!gradres1}
The minimal graded resolution of $A/Q$ as $A$-module is equal to
\begin{equation} \label{eq!resR1}
0  \rightarrow  C_6  \rightarrow  C_5  \rightarrow  C_4  \rightarrow  C_3  \rightarrow  C_2  \rightarrow  C_1  \rightarrow  C_0 \rightarrow  0
\end{equation} 
where
\begin{align*}
  & C_6    =   A(-12),        \quad \quad    C_5 = A(-8)^{6}\oplus A(-9)^{8}   \oplus A(-10)^{6},   \\
  & C_4   =    A(-6)^{8}\oplus A(-7)^{24}\oplus A(-8)^{24}\oplus A(-9)^{8},         \\
  & C_3  =     A(-4)^{3}\oplus A(-5)^{24}\oplus A(-6)^{36}\oplus A(-7)^{24}\oplus  A(-8)^{3},  \\
  &  C_2 =   A(-3)^{8}\oplus A(-4)^{24}\oplus A(-5)^{24}\oplus A(-6)^{8},  \\
  &  C_1 =   A(-2)^{6}\oplus A(-3)^{8}\oplus A(-4)^{6}, \quad \quad   C_0 = A.  
\end{align*}

   Moreover, the canonical module of $A/Q$ is isomorphic to  $(A/Q)(-1)$ and the Hilbert series of  $A/Q$  as graded $A$-module is equal to 
  \[
      \frac{1-6t^2+15t^4-20t^6+15t^{8}-6t^{10}+t^{12}}{(1-t)^8(1-t^2)(1-t^3)}.
  \] 
\end{proposition}

\begin{proof}
The computation of the  minimal graded free resolution  of $A/Q$ is based on the  method which is  described in the proof of  \cite[Proposition~3.4]{NP1}.
Using the minimal graded free resolution (\ref{eq!resR1}) of $A/Q$ and that the sum of the degrees of the variables is equal to $13$ we conclude that 
\[\omega_{A/Q}= A/Q(12-13)= A/Q(-1).\]
The last conclusion of Proposition \ref{prop!gradres1} follows easily from the resolution (\ref{eq!resR1}). 
\end{proof}

By  Propositions \ref{Prop!quasismooth1}, \ref{sing!specsing1} and \ref{prop!gradres1}, it follows that $X$ is a Fano $3$-fold.

\subsection{ Construction of Graded Ring Database entry with ID: 9176} \label{constr2!9176}
In this subsection we sketch the construction of the family of Fano $3$-folds described
in Theorem~\ref{constr!9176}.

Denote by $k=\mathbb{C}$ the field of complex numbers. Consider the 
polynomial ring $R= k[x_i, c_i]$, where $1\leq  i \leq 6$.  Let  $ R_{un}$ be the 
ring in Definition \ref{def!mrin1} and  $\hat{R}=k[x_1,\dots, x_6]$ be the polynomial ring in the variables $x_i$.
We substitute the variables $(c_1, \dots, c_{6})$ which appear in the definitions of the rings $R$ 
and $R_{un}$ with a general element of $k^{6}$ (in the sense of being outside a 
proper Zariski closed subset of $k^{6}$). Let $\hat{I}$ be the ideal of $\hat{R}$ which is 
obtained by the ideal $I$ and $\hat{ I}_{un}$ the ideal of $\hat{R}[T_1,T_2,T_3,T_4]$ which is 
obtained by the ideal $I_{un}$  after this substitution. We set $\hat{R}_{un}= \hat{R}[T_1,T_2,T_3,T_4]/\hat{ I}_{un}$. In 
what follows $x_1, x_3, x_5$ are variables of degree $2$ and $x_2, x_4, x_6$ are variables of degree $3$. Hence, 
from the discussion before the Proposition \ref{prop!hom1} it follows that the 
degrees of $T_2, T_3, T_4$ are equal to $2$ and the degree of $T_1$ is equal to $4$. According to 
this grading the ideals $\hat{I}$ and $\hat{ I}_{un}$ are homogeneous. Due to 
Theorem~\ref{main!thm1},   $\text{Proj} \   \hat{R}_{un}\subset \mathbb{P} (2^{6}, 3^{3}, 4)$ is a projectively Gorenstein $3$-fold.

Let $A= k[w_{1}, w_{2},x_1,x_5, T_2, T_3, T_4,x_2,x_4,x_6]$  be the polynomial ring 
over $k$ with $w_{1}, w_{2}$ variables of degree $1$ and the other variables with
degree noted as above. Consider the unique $k$-algebra homomorphism
\begin{center}
$\psi\colon \hat{R}[T_1, T_2, T_3, T_4]\rightarrow A$
\end{center}
such that
\begin{center}
$\psi(x_1)= x_1$, \, $\psi(x_2)= x_2$, \,   $\psi(x_3)= f_1$, \, $\psi(x_4)= x_4$,
\end{center}
\begin{center}
$\psi(x_5)= x_5$, \, $\psi(x_6)= x_6$, \, $\psi(T_1)= f_2$, \,  $\psi(T_2)= T_2$,
\end{center}
\begin{center}
$\psi(T_3)= T_3$, \, $\psi(T_4)= T_4$
\end{center}
where,

$f_1= l_1w_1^2+l_2w_1w_2+l_3w_2^2+l_4x_1+l_5x_5+l_6T_2+l_7T_3+l_8T_4$,

$f_2=l_9w_1^4+l_{10}w_1^3w_2+l_{11}w_1^2w_2^2+l_{12}w_1w_2^3+l_{13}w_2^4+l_{14}w_1^2x_1+l_{15}w_1w_2x_1+l_{16}w_2^2x_1
+l_{17}x_1^2+l_{18}w_1^2x_5+l_{19}w_1w_2x_5+l_{20}w_2^2x_5+l_{21}x_1x_5+l_{22}x_5^2+l_{23}w_1^2T_2+l_{24}w_1w_2T_2+l_{25}w_2^2T_2+
l_{26}x_1T_2+l_{27}x_5T_2+l_{28}T_2^2+l_{29}w_1^2T_3+l_{30}w_1w_2T_3+l_{31}w_2^2T_3+l_{32}x_1T_3+l_{33}x_5T_3+l_{34}T_2T_3+l_{35}T_3^2+l_{36}w_1^2T_4+l_{37}w_1w_2T_4+l_{38}w_2^2T_4+l_{39}x_1T_4+l_{40}x_5T_4+l_{41}T_2T_4+l_{42}T_3T_4+l_{43}T_4^2+l_{44}w_1x_2+l_{45}w_2x_2+l_{46}w_1x_4+l_{47}w_2x_4+
l_{48}w_1x_6+l_{49}w_2x_6$,

\noindent and   $(l_1,\dots, l_{49})\in k^{49}$ are general. In other words, $f_1$ is a general degree $2$ homogeneous  element of $A$ and $f_2$ is a general degree $4$ homogeneous  element of $A$ .

Denote by $Q$ the ideal of the ring A generated by the subset   $\psi(\hat{I}_{un})$.

Let $X= V(Q)\subset \mathbb{P} (1^{2}, 2^5, 3^3)$. It is immediate that  $X\subset \mathbb{P}(1^{2}, 2^5, 3^3)$ is a codimension $6$ projectively Gorenstein $3$-fold.

\begin{proposition} \label{constr!primeno2}
The ring  $A/Q$ is an integral domain.
\end{proposition}

\begin{proof}
It is enough to show that the ideal $Q$ is prime. For a specific choice of rational
values for the parameters $c_i,  l_j$, for $1\leq i\leq 6$ and $1\leq j\leq 49$ we checked using the computer
algebra program Macaulay2 that the ideal which was obtained by $Q$ is a homogeneous, codimension~$6$, prime ideal with the right Betti table. 
\end{proof}

In what follows, we show that the only singularities of $X\subset \mathbb{P}(1^2, 2^5, 3^3)$  are eight quotient singularities of type $\frac{1}{2}(1,1,1)$. According to the discussion after  
 \cite [Definition~2.7]{PV}, $X$ belongs to the Mori category. 

The proof of the following proposition is based on a computation with computer algebra system Singular \cite{GPS01} using the strategy described in \cite [Proposition~6.4]{PV} and is omitted.

\begin{proposition}\label{Prop!quasismooth2}
Consider   $X= V(Q)\subset \mathbb{P} (1^2, 2^5, 3^3)$. Denote by  $X_{cone}\subset \mathbb{A}^{10}$ the affine cone over $X$. 
The scheme  $X_{cone}$ is smooth outside the vertex of the cone.
\end{proposition}

\begin{remark}
For the computation of singular locus of weighted projective space in Proposition \ref{sing!specsing2},  we follow \cite [Section~5]{IF}. 
\end{remark}

\begin{proposition} \label{sing!specsing2}
Consider the singular locus 
\[
\text{Sing}( \mathbb{P}(1^{2}, 2^5, 3^3))= F_1\cup F_2\]
\newpage
where,
\[F_1= \{[0:0:a:b:c:d:e:0:0:0]: [a:b:c:d:e]\in \mathbb{P}^4 \}\]
and
\[F_2=\{ [0:0:0:0:0:0:0:a:b:c]: [a:b:c]\in \mathbb{P}^2\}\]
of the weighted projective space  $\mathbb{P}(1^{2}, 2^5, 3^3)$. The intersection of $X$ with $\text{Sing}( \mathbb{P}(1^{2}, 2^5, 3^3))$ consists of eight reduced points which are quotient singularities of type $\frac{1}{2}(1,1,1)$ for $X$.
\end{proposition}

\begin{proof}
We proved with the computer algebra program Macaulay2 that the intersection of $X$ with $Z$ consists of eight reduced points. Following the strategy of the proof of Proposition \ref{sing!specsing1}, we checked that each of these points is a quotient singularity of type $\frac{1}{2}(1,1,1)$. 
\end{proof}

\begin{lemma}\label{lem!canonmod2}
Let $\omega_{\hat{R}/\hat{I}}$ be the canonical module of $\hat{R}/\hat{I}$. It holds that  the canonical module $\omega_{\hat{R}/\hat{I}}$ is isomorphic to $\hat{R}/\hat{I}(-5)$.
\end{lemma}

\begin{proof}
From the minimal graded free resolution of $\hat{R}/\hat{I}$ as  $\hat{R}$-module
\[
0\rightarrow \hat{R}(-10)\rightarrow \hat{R}(-5)^{2}\rightarrow \hat{R}  
\]
and the fact that the sum of the degrees of the variables is equal to $15$ we conclude that 
\[\omega_{\hat{R}/\hat{I}}= \hat{R}/\hat{I}(10-15)= \hat{R}/\hat{I}(-5).\] 
\end{proof}

\begin{proposition} \label{prop!gradres2}
The minimal graded resolution of $A/Q$ as $A$-module is equal to
\begin{equation} \label{eq!resR2}
0  \rightarrow  C_6  \rightarrow  C_5  \rightarrow  C_4  \rightarrow  C_3  \rightarrow  C_2  \rightarrow  C_1  \rightarrow  C_0 \rightarrow  0
\end{equation} 
where
\begin{align*}
  & C_6    =   A(-20),        \quad \quad    C_5 = A(-14)^{6}\oplus A(-15)^{8}   \oplus A(-16)^{6},   \\
  & C_4   =    A(-11)^{8}\oplus A(-12)^{24}\oplus A(-13)^{24}\oplus A(-14)^{8},         \\
  & C_3  =     A(-8)^{3}\oplus A(-9)^{24}\oplus A(-10)^{36}\oplus A(-11)^{24}\oplus  A(-12)^{3},  \\
  &  C_2 =   A(-6)^{8}\oplus A(-7)^{24}\oplus A(-8)^{24}\oplus A(-9)^{8},  \\
  &  C_1 =   A(-4)^{6}\oplus A(-5)^{8}\oplus A(-6)^{6}, \quad \quad   C_0 = A.  
\end{align*}

   Moreover, the canonical module of $A/Q$ is isomorphic to  $(A/Q)(-1)$ and the Hilbert series of  $A/Q$  as graded $A$-module is equal to 
  \[
      \frac{1-6t^4-8t^5+2t^6+24t^7+21t^8-16t^9-36t^{10}-16t^{11}+21t^{12}+24t^{13}+2t^{14}-8t^{15}-6t^{16}+t^{20}}{(1-t)^2(1-t^2)^5(1-t^3)^3}.
  \] 
\end{proposition}

\begin{proof}
The computation of the  minimal graded free resolution  of $A/Q$ is based on the  method which is  described in the proof of  \cite[Proposition~3.4]{NP1}.
Using the minimal graded free resolution (\ref{eq!resR2}) of $A/Q$ and that the sum of the degrees of the variables is equal to $21$ we conclude that 
\[\omega_{A/Q}= A/Q(20-21)= A/Q(-1).\]
The last conclusion of Proposition \ref{prop!gradres2} follows easily from the resolution (\ref{eq!resR2}). 
\end{proof}

By  Propositions \ref{Prop!quasismooth2}, \ref{sing!specsing2} and \ref{prop!gradres2}, it follows that $X$ is a Fano $3$-fold.

\subsection{ Construction of Graded Ring Database entry with ID: 24198} \label{constr3!24198}
In this final subsection, we sketch the construction for the family of  Fano 3-folds which is described
in Theorem \ref{constr!24198}.

Denote by $k=\mathbb{C}$ the field of complex numbers. Consider the 
polynomial ring $R= k[x_i, c_i]$, where $1\leq  i \leq 6$.  Let  $ R_{un}$ be the 
ring in Definition \ref{def!mrin1} and  $\hat{R}=k[x_1,\dots, x_6,c_3,c_6]$ be the polynomial ring in the variables $x_i$ and $c_3,c_6$.
We substitute the variables $(c_1,c_2,c_4,c_5)$ which appear in the definitions of the rings $R$ 
and $R_{un}$ with a general element of $k^{4}$ (in the sense of being outside a 
proper Zariski closed subset of $k^{4}$).  Let $\hat{I}$ be the ideal of $\hat{R}$ which is 
obtained by the ideal $I$ and $\hat{ I}_{un}$ the ideal of $\hat{R}[T_1,T_2,T_3,T_4]$ which is 
obtained by the ideal $I_{un}$  after this substitution. We set $\hat{R}_{un}= \hat{R}[T_1,T_2,T_3,T_4]/\hat{ I}_{un}$. In 
what follows $x_1, x_3, x_5, x_6, c_3, c_6$ are variables of degree $1$ and $x_2, x_4$ are variables of degree $2$. Hence, 
from the discussion before the Proposition \ref{prop!hom1} it follows that the degree of $T_1$ is equal to $3$, the
degrees of $T_2, T_3$ are equal to $2$ and the degree of $T_4$ is equal to $1$. According to 
this grading the ideals $\hat{I}$ and $\hat{ I}_{un}$ are homogeneous. Due to 
Theorem~\ref{main!thm1},   $\text{Proj} \   \hat{R}_{un}\subset \mathbb{P} (1^{7}, 2^{4}, 3)$ is a projectively Gorenstein $5$-fold.

Let $A= k[x_{1}, x_{3},x_5,x_6, c_3, c_6, x_2,x_4,T_3,T_1]$  be the polynomial ring 
with variables of degree noted as above. Consider the unique $k$-algebra homomorphism
\begin{center}
$\psi\colon \hat{R}[T_1, T_2, T_3, T_4]\rightarrow A$
\end{center}
such that
\begin{center}
$\psi(x_1)= x_1$, \, $\psi(x_2)= x_2$, \,   $\psi(x_3)= x_3$, \, $\psi(x_4)= x_4$,
\end{center}
\begin{center}
$\psi(x_5)= x_5$, \, $\psi(x_6)= x_6$, \, $\psi(c_3)= c_3$, \,  $\psi(c_6)= c_6$,
\end{center}
\begin{center}
$\psi(T_1)= T_1$, \, $\psi(T_2)= f_1$,  \, $\psi(T_3)= T_3$, \, $\psi(T_4)= f_2$
\end{center}
where,

$f_1=l_1x_1^2+l_2x_1x_3+l_3x_3^2+l_4x_1x_5+l_5x_3x_5+l_6x_5^2+l_7x_1x_6+l_8x_3x_6+l_9x_5x_6+l_{10}x_6^2+l_{11}x_1c_3+l_{12}x_3c_3+l_{13}x_5c_3+l_{14}x_6c_3+l_{15}c_3^2+l_{16}x_1c_6+l_{17}x_3c_6+l_{18}x_5c_6+l_{19}x_6c_6+l_{20}c_3c_6+l_{21}c_6^2+l_{22}x_2+l_{23}x_4+l_{24}T_3$,

$f_2=l_{25}x_1+l_{26}x_3+l_{27}x_5+l_{28}x_6+l_{29}c_3+l_{30}c_6$,

\noindent and   $(l_1,\dots, l_{30})\in k^{30}$ are general. In other words, $f_1$ is a general degree $2$ homogeneous  element of $A$ and $f_2$ is a general degree $1$ homogeneous  element of $A$.

Denote by $Q$ the ideal of the ring A generated by the subset   $\psi(\hat{I}_{un})$.

Let $X= V(Q)\subset \mathbb{P} (1^{6}, 2^3, 3)$. It is immediate that  $X\subset \mathbb{P}(1^{6}, 2^3, 3)$ is a codimension $6$ projectively Gorenstein $3$-fold.

\begin{proposition} \label{constr!primeno2}
The ring  $A/Q$ is an integral domain.
\end{proposition}

\begin{proof}
It is enough to show that the ideal $Q$ is prime. For a specific choice of rational
values for the parameters $c_i,  l_j$, for $i\in \{1,2,4,5\}$ and $1\leq j\leq 30$ we checked using the computer
algebra program Macaulay2 that the ideal which was obtained by $Q$ is a homogeneous, codimension~$6$, prime ideal with the right Betti table. 
\end{proof}

In what follows, we show that the only singularities of $X\subset \mathbb{P}(1^6, 2^3, 3)$  are two quotient singularities of type $\frac{1}{2}(1,1,1)$ and one quotient singularity of type $\frac{1}{3}(1,1,2)$. According to the discussion after \cite [Definition~2.7]{PV}, $X$ belongs to the Mori category. 

The proof of the following proposition is based on a computation with computer algebra system Singular \cite{GPS01} using the strategy described in \cite [Proposition~6.4]{PV} and is omitted.

\begin{proposition}\label{Prop!quasismooth3}
Consider   $X= V(Q)\subset \mathbb{P} (1^6, 2^3, 3)$. Denote by  $X_{cone}\subset \mathbb{A}^{10}$ the affine cone over $X$. 
The scheme  $X_{cone}$ is smooth outside the vertex of the cone.
\end{proposition}

\begin{remark}
For the computation of singular locus of weighted projective space in Proposition \ref{sing!specsing3},  we follow \cite [Section~5]{IF}. 
\end{remark}

\begin{proposition} \label{sing!specsing3}
Consider the singular locus 
\[
\text{Sing}( \mathbb{P}(1^{6}, 2^3, 3))= F_1\cup \{ [0:0:0:0:0:0:0:0:0:1]\}\]
where,
\[F_1=\{  [0:0:0:0:0:0:a:b:c:0]: [a:b:c]\in \mathbb{P}^2\}\]
of the weighted projective space  $\mathbb{P}(1^{6}, 2^3, 3)$. The intersection of $X$ with $\text{Sing}( \mathbb{P}(1^{6}, 2^3, 3))$  consists of two reduced points which are quotient singularities of type $\frac{1}{2}(1,1,1)$ and one reduced point which is quotient singularity of type $\frac{1}{3}(1,1,2)$ for $X$.
\end{proposition}

\begin{proof}
We proved with the computer algebra program Macaulay2 that the intersection of $X$ with $Z$ consists of three reduced points. Following the strategy of the proof of Proposition~\ref{sing!specsing1}, we checked that two of these points are quotient singularities of type $\frac{1}{2}(1,1,1)$ and the third point is quotient singularity of type $\frac{1}{3}(1,1,2)$ for $X$. 
\end{proof}

\begin{lemma}\label{lem!canonmod3}
Let $\omega_{\hat{R}/\hat{I}}$ be the canonical module of $\hat{R}/\hat{I}$. It holds that  the canonical module $\omega_{\hat{R}/\hat{I}}$ is isomorphic to $\hat{R}/\hat{I}(-4)$.
\end{lemma}

\begin{proof}
From the minimal graded free resolution of $\hat{R}/\hat{I}$ as  $\hat{R}$-module
\[
0\rightarrow \hat{R}(-6)\rightarrow \hat{R}(-3)^{2}\rightarrow \hat{R}  
\]
and the fact that the sum of the degrees of the variables is equal to $10$ we conclude that 
\[\omega_{\hat{R}/\hat{I}}= \hat{R}/\hat{I}(6-10)= \hat{R}/\hat{I}(-4).\] 
\end{proof}

\begin{proposition} \label{prop!gradres3}
The minimal graded resolution of $A/Q$ as $A$-module is equal to
\begin{equation} \label{eq!resR3}
0  \rightarrow  C_6  \rightarrow  C_5  \rightarrow  C_4  \rightarrow  C_3  \rightarrow  C_2  \rightarrow  C_1  \rightarrow  C_0 \rightarrow  0
\end{equation} 
where
\begin{align*}
  & C_6    =   A(-14),        \quad \quad    C_5 = A(-9)^{2}\oplus A(-10)^{7}   \oplus A(-11)^{10}  \oplus  A(-12)^{1},   \\
  & C_4   =    A(-7)^{4}\oplus A(-8)^{20}\oplus A(-9)^{28}\oplus A(-10)^{12},         \\
  & C_3  =     A(-5)^{2}\oplus A(-6)^{25}\oplus A(-7)^{36}\oplus A(-8)^{25}\oplus  A(-9)^{2},  \\
  &  C_2 =   A(-4)^{12}\oplus A(-5)^{28}\oplus A(-6)^{20}\oplus A(-7)^{4},  \\
  &  C_1 =   A(-2)^{1}\oplus A(-3)^{10}\oplus A(-4)^{7}\oplus A(-5)^{2}, \quad \quad   C_0 = A.  
\end{align*}

   Moreover, the canonical module of $A/Q$ is isomorphic to  $(A/Q)(-1)$ and the Hilbert series of  $A/Q$  as graded $A$-module is equal to 
  \[
   \frac{1-t^2-10t^3+5t^4+24t^5-5t^6-28t^7-5t^8+24t^9+5t^{10}-10t^{11}-t^{12}+t^{14}}{(1-t)^6(1-t^2)^3(1-t^3)}.
  \] 
\end{proposition}

\begin{proof}
The computation of the  minimal graded free resolution  of $A/Q$ is based on the  method which is  described in the proof of  \cite[Proposition~3.4]{NP1}.
Using the minimal graded free resolution (\ref{eq!resR2}) of $A/Q$ and that the sum of the degrees of the variables is equal to $15$ we conclude that 
\[\omega_{A/Q}= A/Q(14-15)= A/Q(-1).\]
The last conclusion of Proposition \ref{prop!gradres3} follows easily from the resolution (\ref{eq!resR3}). 
\end{proof}

By  Propositions \ref{Prop!quasismooth3}, \ref{sing!specsing3} and \ref{prop!gradres3}, it follows that $X$ is a Fano $3$-fold.

\section* {Acknowledgements} 

\label{sec!acknowledgements}
I would like to thank Stavros Papadakis for important discussions and suggestions which have improved the present paper.  I benefited from experiments with the computer algebra programs Macaulay2~\cite{GS} and Singular~\cite{GPS01}. Part of this work is contained in my PhD Thesis,~\cite{PV1} carried out at the University of Ioannina, Greece. This work was financially supported by Horizon Europe ERC Grant number: 101045750 / Project acronym: HodgeGeoComb with principal investigator Karim Adiprasito, whom I warmly thank.


\begin{thebibliography}{99}

\bibitem{AL}
Alt{\i}nok, S. (1998). Graded rings corresponding to polarised K3 surfaces and Q-Fano 3-folds. PhD dissertation. University of Warwick, Coventry, UK.




\bibitem{ABR}
 Alt{\i}nok, S., Brown, G., Reid, M. (2002). Fano 3-folds, K3 surfaces and graded rings. Topology and geometry: commemorating SISTAG.  Contemp. Math. Amer. Math. Soc. 314: 25--53.  Providence, RI. 








\bibitem{BP1}
 B\"{o}hm, J., Papadakis, S.A. (2012). On the structure of Stanley-Reisner rings associated to cyclic polytopes.
Osaka J. Math. 49:  81--100.




\bibitem{BP2}
 B\"{o}hm, J., Papadakis, S.A. (2013). Stellar subdivisions and Stanley-Reisner rings of Gorenstein complexes. 
Australas. J. Combin. 55: 235--247.




\bibitem{BP3}
B\"{o}hm, J., Papadakis, S.A. (2015). Bounds for the Betti numbers of successive stellar subdivisions of a simplex. Hokkaido Math. J. 44: 341--364.



\bibitem{BP4}
B\"{o}hm, J., Papadakis, S.A.  (2020). Weak Lefschetz Property and Stellar Subdivisions of Gorenstein Complexes. Australas. J. Combin. 76: 266--287.



\bibitem{BGR}
 Brown, G., Reid, M. (2017). Diptych varieties. II: Apolar varieties. Higher dimensional algebraic geometry—in honour of Professor Yujiro Kawamata's sixtieth birthday. Adv. Stud. Pure Math. 74: 41--72.





\bibitem{BH}
Bruns, W., Herzog, J. (1993). Cohen-Macaulay rings. Cambridge, Cambridge Studies in Advanced Mathematics, 39. Cambridge University Press. 


\bibitem{BR}
Brown, G.,  Kasprzyk, A M. (2002).  Graded ring database. Online searchable database. Available at:

http://grdb.co.uk/. 




\bibitem{BG}
Brown, G., Georgiadis, K. (2017). Polarized Calabi-Yau 3-folds in codimension 4. Math. Nachr. 290: 710--725.





\bibitem{BKR}
Brown, G., Kerber,  M., Reid, M. (2012). Fano 3-folds in codimension 4, Tom and Jerry, Part I. Compos. Math. 148: 1171--1194.






\bibitem{BS1}
Brown, G., Suzuki, K. (2007). Fano 3-folds with divisible anticanonical class. Manuscripta. Math. 123: 37--51.


\bibitem{BS2}
Brown, G., Suzuki, K. (2007). Computing certain Fano 3-folds. Japan J. Indust. Appl. Math. 24: 241--250.


\bibitem{BS3}
Brown, G., Kasprzyk, A M. (2022). Kawamata boundedness for Fano threefolds
and the graded ring database. arXiv preprint, available at \href{https://arxiv.org/abs/2201.07178}{https://arxiv.org/abs/2201.07178}.




















\bibitem{CL1}
Campo, L. (2020). Sarkisov links for index 1 fano 3-folds in codimension 4. arXiv preprint, available at \href{https://arxiv.org/abs/2011.12209}{https://arxiv.org/abs/2011.12209}.




\bibitem{CL2}
Campo, L. (2021). Fano 3-folds and double covers by half elephants. arXiv preprint, available at \href{https://arxiv.org/abs/2103.17219}{https://arxiv.org/abs/2103.17219}.







\bibitem{CM}
Corti, A., Mella, M. (2004). Birational geometry of terminal quartic 3-folds I. Amer. J. Math. 126: 739--761.






\bibitem{CPR}
Corti, A.,  Pukhlikov, A., Reid, M. (2000). Fano 3-fold hypersurfaces. Explicit birational geometry of 3-folds. Cambridge, 
 London Math. Soc. Lecture Note Ser., 281, Cambridge Univ. Press, 175--258.



\bibitem{GPS01}
 Decker, W., Greuel, G.-M., Pfister, G.,  Sch{\"o}nemann, H. (2019).  Singular 4-1-2 --- A computer algebra system for polynomial computations. Available at:

http://www.singular.uni-kl.de.


\bibitem{E}
Eisenbud, D. (1995). Commutative algebra with a view toward algebraic geometry.  Graduate Texts in Mathematics, 150. Springer-Verlag.



\bibitem{GS} 
Grayson, D., Stillman, M.  Macaulay2, a software system for research in algebraic geometry. Available at:

https://macaulay2.com/.

\bibitem{IF}
Iano-Fletcher, A. R. (2000). Working with weighted complete intersections.  Cambridge, Cambridge Studies in Advanced Mathematics, 281. Cambridge University Press: 101--173. 


\bibitem{KM}
 Kustin, A., Miller, M. (1983). Constructing big Gorenstein ideals from small ones.  J. Algebra 85: 303--322.



\bibitem{KM1}
Kustin, A., Miller, M. (1980). Algebra structures on minimal resolutions of Gorenstein rings of embedding codimension four. Math. Z. 173: 171--184.


\bibitem{KM2}
Kustin, A., Miller, M. (1981). A general resolution for grade four Gorenstein ideals. Manuscripta Math. 35: 221--269.


\bibitem{KM3}
Kustin, A., Miller, M. (1892). Structure theory for a class of grade four Gorenstein ideals. Trans. Amer. Math. Soc. 270: 287--307.


\bibitem{KM4}
Kustin, A., Miller, M. (1984). Deformation and linkage of Gorenstein algebras. Trans. Amer. Math. Soc. 284: 501--534.


\bibitem{KM5}
Kustin, A., Miller, M. (1985). Classification of the Tor-algebras of codimension four Gorenstein local rings. Math. Z. 190: 341--355.




\bibitem{LP}
 Liedtke, C., Papadakis, S.A (2010). Birational modifications of surfaces via unprojections. J. Algebra 323: 2510--2519.





\bibitem{NP1}
Neves, J., Papadakis, S.A. (2009). A construction of numerical Campedelli surfaces with torsion $\mathbb{Z}/6$. Trans. Amer. Math. Soc. 361: 4999--5021.




\bibitem{NP2}
Neves, J., Papadakis, S.A. (2013). Parallel Kustin-Miller unprojection with an application to Calabi--Yau geometry. Proc. Lond. Math. Soc. (3) 106: 203--223.




\bibitem{P1} 
 Papadakis, S. A. (2001). Gorenstein rings and Kustin-Miller unprojection. PhD dissertation. University of Warwick, Coventry, UK. Available at:

https://www.math.tecnico.ulisboa.pt/$\sim$ papadak/.





\bibitem{P2} 
 Papadakis, S.A. (2004). Kustin-Miller unprojection with complexes. J. Algebraic Geom. 13: 249--268.


\bibitem{P3} 
 Papadakis, S.A. (2006). Type II unprojection. J. Algebraic Geom. 15: 399–414.




\bibitem{P4} 
 Papadakis, S.A. (2007). Towards a general theory of unprojection. J. Math. Kyoto Univ. 47: 579–598.



\bibitem{PR}
Papadakis, S.A, Reid, M. (2004). Kustin-Miller unprojection without complexes. J. Algebraic Geom. 13: 563--577.

\bibitem{PV}
Petrotou, V. (2022). Tom \& Jerry triples with an application to Fano 3-folds. Commun. Algebra 50: 3960--3977.

\bibitem{PV1} 
Petrotou, V. (2022). Unprojection Theory, Applications to Algebraic Geometry and Anisotropy of Simplicial Spheres. PhD dissertation. University of Ioannina, Greece. Available at:

https://sites.google.com/view/vpetrotou.


\bibitem{R1}
Reid, M. (2000). Graded Rings and Birational Geometry. In: Ohno, K., ed.  Proc. of algebraic symposium (Kinosaki, Oct 2000): 1--72. 
Available at:  

 \href{https://www.maths.warwick.ac.uk/~miles/3folds}{https://www.maths.warwick.ac.uk/$\sim$miles/3folds}. 






\bibitem{TR}
Taylor, R. (2020). Type II unprojections, Fano threefolds and codimension four constructions.Phd Thesis, University of Warwick.



\end{thebibliography}
\end{document}